\documentclass[12pt]{article}
\usepackage{mathtext,amssymb,amsmath}
\usepackage[T1,T2A]{fontenc}
\usepackage[cp1251]{inputenc}
\usepackage[english]{babel}
\usepackage{xcolor}
\def\dom{\mathop{\rm dom}}

\def\epsilon{\varepsilon}
\def\phi{\varphi}

\def\dom{\mathop{\rm dom}}

\def\dom{\mathop{\rm dom}}

\def\ball{{I\kern -.35em B}}
\def\epsilon{\varepsilon}
\def\phi{\varphi}

\def\reals{{I\kern-.35em R}} \def\Reals{\overline{I\kern-.35em R}}

\def\qed{\hfill{$\vcenter{\hrule height1pt \hbox{\vrule width1pt height5pt
    \kern5pt \vrule width1pt} \hrule height1pt}$} \medskip}

\def\text#1{\,\;\hbox{#1}\;\,}

\newcommand{\R}{\mathbb{R}}

\newcommand{\beq}{\begin{equation}}
\newcommand{\eeq}{\end{equation}}

\newenvironment{proof}
         {\begin{trivlist}\item[
         {\bf Proof.}]}{{\hfill $\square$} \end{trivlist}}

\newtheorem{theo}{Theorem}[section]
\newtheorem{defi}[theo]{Definition}
\newtheorem{rem}[theo]{Remark}

\begin{document}
\title{\LARGE\bf Saddle points in completely regular topological spaces\thanks{The research of the first and the third named authors is financed by the European Union-NextGenerationEU, through the National Recovery and Resilience Plan of the Republic of Bulgaria, project  BG-RRP-2.004-0008-C01}}

\author{Detelina Kamburova\thanks{Institute of Mathematics and Informatics, Bulgarian Academy of Sciences, Acad. G. Bonchev Str., Block 8,
1113 Sofia, Bulgaria, e-mail: detelinak@math.bas.bg}, Rumen Marinov\thanks{Department of Mathematics and Physics, Technical University of Varna, 1, Studentska Str.,
9010 Varna, Bulgaria, e-mail: marinov\_r@yahoo.com}, Nadia Zlateva\thanks{Faculty of Mathematics and Informatics, Sofia University,   5, James Bourchier Blvd, 1164 Sofia, Bulgaria, e-mail: zlateva@fmi.uni-sofia.bg}}
\date{}

\maketitle

\begin{abstract}
We give a characterization of completely regular topological spaces. Applying some recent results for supinf problems in completely regular topological spaces we establish a variational principle for saddle points. Well-posedness of saddle point problems is studied as well. \\[0.2cm]
\end{abstract}

\section{Introduction}

Variational principles concern sufficient conditions under which after a perturbation of the optimized function, the perturbed minimization problem has a solution. In 1972  Ekeland~\cite{Ekeland} proved a variational principle in complete metric spaces. This principle has been shown to be equivalent to the completeness of the metric space, see  \cite{Sullivan}, as well as to Caristi's fixed point theorem  and Takahashi's minimization principle, see \cite{Oettli_Thera}. McLinden \cite{McLinden} used the Ekeland variational principle to derive some results of minimax type in Banach spaces.

\smallskip

In 2010 Kenderov and Revalski~\cite{Kend_Rev_2010} proved a variational principle in completely regular topological spaces. Later, in \cite{Kend_Rev_2017}, they gave a sufficient condition for existence of a perturbation such that the perturbed problem is (Tykhonov) well-posed. Let us recall that the problem to minimize $f:X \rightarrow \R \cup \{+\infty\}$, where $X$ is a topological space, is called   well-posed if it has  unique solution $x_0 \in X$ and for every minimizing sequence $\{x_n\}_n \subset X$, $f(x_n) \rightarrow \inf_X f$ it holds that $x_n \rightarrow x_0$, i.e. $x_0$ is the strong minimum of $f$ on $X$. In the same paper they obtained a variational principle for supinf problems.

\smallskip

We prove that if the variational principle of Kenderov and Revalski holds for every proper lower semicontinuous bounded below function defined on a topological space then the space is completely regular (Theorem~\ref{characterization1}) and if the strong variational principle holds, then additionally the space satisfies the first axiom of countability (Theorem~\ref{characterization2}). A minimax variational principle in completely regular topological spaces is proved in Theorem~\ref{saddle_point}. In the final section we give a sufficient condition for existence of a perturbation such that the perturbed saddle point problem is well-posed. At the end, we give a characterization of well-posed perturbed problems.

\section{Characterization of completely regular topological spaces}
A topological space $X$ is said to be completely regular if it is Hausdorff and for every set $A \subset X$ and a point $x \in X \setminus \overline{A}$ (i.e. not belonging to the closure of $A$) there exists a continuous function $f: X \rightarrow \mathbb{R}$ such that $f(x) \notin \overline{f(A)}$, see  \cite{Cullen}. The following assertion is often used as a definition for completely regular topological spaces: a Hausdorff topological space $X$ is completely regular if for every point $x$ and a closed set $A$, such that $x \notin A$ there exists a continuous function $f: X \rightarrow [0,1]$ such that $f(x)=0$ and $f(A)=1$.
Let $X$ be a completely regular topological space and $f: X \rightarrow [-\infty, +\infty]$ be an extended real-valued function. The domain of $f$ is denoted by $\dom f$ and consists of all points in $X$ at which $f$ has a finite value. The function $f$ is  proper if its domain is not empty. Denote by $C(X)$ the space of all continuous and bounded real-valued functions defined on $X$. The space $C(X)$ equipped with the supremum norm $\|f\|_{\infty}:=\sup\{|f(x)|: x \in X\}$ is a real Banach space. We denote by $\mathbb R_{+}$ the set of all nonnegative real numbers. 

Let us recall the variational principle of Kenderov and Revalski:
\begin{theo}[\cite{Kend_Rev_2010}]\label{basic_lemma1}   Let $X$ be a completely regular topological space and $f: X \rightarrow \mathbb R \cup \{+\infty\}$ be a proper lower semicontinuous function bounded from below. Let $x_0 \in \dom f$ and $\varepsilon>0$ be such that $f(x_0)<\inf_{X} f + \varepsilon$. Then, there exists a continuous bounded function $h: X \rightarrow \mathbb R_+$, $h(x_0)=0$, $\|h\|_{\infty} < \varepsilon$ and the function $f+h$ attains its minimum in $X$ at $x_0$. Moreover, $h$ can be chosen such that $\|h\|_{\infty}=f(x_0)-\inf_{X}f$.
\end{theo}

Assuming that the variational principle of Kenderov and Revalski holds in a Hausdorff topological space $X$, we prove that $X$ is completely regular.

\begin{theo}\label{characterization1}
Let $X$ be a Hausdorff topological space. If for every function $f: X \rightarrow \mathbb R \cup \{+\infty\}$ which is proper lower semicontinuous and bounded from below and for every $x_0 \in \dom f$, there exists a continuous bounded function $h: X \rightarrow \mathbb R_+$, $h(x_0)=0$, $\|h\|_{\infty} = f(x_0) - \inf_X f$ such that the function $f+h$ attains its minimum on $X$ at $x_0$, then $X$ is a completely regular topological space.
\end{theo}

\begin{proof} Let $x_0 \in X$, $A$ be a closed subset of $X$, $x_0 \notin A$ and let
\[f(x):=\left\{\begin{array}{ll}
0, & \text{if } x \in A,\\
1, & \text{otherwise.}\\
\end{array}\right.\]
Since $f$ is a proper lower semicontinuous function bounded from below then, by assumption there exists $h: X \rightarrow \mathbb R_+$, $h(x_0)=0$, $\|h\|_{\infty} = 1$ such that the function $f+h$ attains its minimum in $X$ at $x_0$. Let $x \in A$ be arbitrary. Then
\[f(x)+h(x)=h(x) \geq f(x_0)+h(x_0)=f(x_0)=1,\]
so $h(x) \geq 1$. As $\|h\|_{\infty} \leq 1$, therefore $h(A) \equiv 1$ and $X$ is a completely regular topological space.
\end{proof}

\begin{rem} From the proof of Theorem \ref{characterization1} it is clear that it is sufficient the variational principle to hold only for characteristic functions of closed sets in the Hausdorff space $X$ to ensure that $X$ is completely regular.
\end{rem}

Let us recall now the strong variational principle of Kenderov and Revalski.

\begin{theo}[\cite{Kend_Rev_2017}]\label{basic_lemma2}  
Let $X$ be a completely regular space and $f: X \rightarrow \mathbb R \cup \{+\infty\}$ be a proper lower semicontinuous function bounded from below. Let $x_0 \in \dom f$ has a countable local base in $X$. Let $\varepsilon>0$ be arbitrary. Then, there exist a continuous bounded function
\[h: X \rightarrow \mathbb R_+, \;  h(x_0)=0, \; \|h\|_{\infty} < f(x_0)-\inf_{X}f+\varepsilon,\]
such that the function $f+h$ attains its strong minimum on $X$ at $x_0$.
\end{theo}

Assuming that the strong variational principle of Kenderov and Revalski holds in a Hausdorff topological space $X$, we will prove that $X$ is a completely regular topological space that satisfies the first axiom of countability.

\begin{theo} \label{characterization2} 
Let $X$ be a Hausdorff topological space. If for every function $f: X \rightarrow \mathbb R \cup \{+\infty\}$ which is proper lower semicontinuous and bounded from below, for every $x_0 \in \dom f$ and for every $\varepsilon>0$, there exists a continuous bounded function
\[h: X \rightarrow \mathbb R_+, \;  h(x_0)=0, \; \|h\|_{\infty} < f(x_0)-\inf_{X}f+\varepsilon,\]
such that the function $f+h$ attains its strong minimum on $X$ at $x_0$, then $X$ is a completely regular topological space that satisfies the first axiom of countability.
\end{theo}

\begin{proof}
Let $x_0 \in X$, $A$ be a subset of $X$, $x_0 \notin \overline{A}$ and let
\[f(x):=\left\{\begin{array}{ll}
0, & \text{if } x \in \overline{A},\\
1, & \text{otherwise.}\\
\end{array}\right.\]
Since $f$ is a proper lower semicontinuous function bounded from below then, by assumption, there exists $h: X \rightarrow \mathbb R_+$, $h(x_0)=0$, such that the function $f+h$ attains its strong minimum on $X$ at $x_0$. Let $x \in \overline{A}$ be arbitrary. Then
\[h(x)= f(x)+h(x) > f(x_0)+h(x_0)=f(x_0)=1,\]
so $h(x) > 1$ for all $x \in \overline{A}$. Therefore, $\overline{h(\overline{A})}\ge 1$. Since $h(x_0)=0$, $h(x_0) \notin \overline{h(\overline{A})}\supset \overline{h( A)}$ and $X$ is a completely regular topological space.

\smallskip

Now we will consider the constant function  $f(x)=0$ for all $x \in X$ and arbitrary fixed $x_0\in X$. By assumption there exists a continuous bounded function $h$ which attains its strong minimum at $x_0$ and $\min_X h=h(x_0)=0$. Consider the sets
\[L_n:=\{x \in X: h(x)<1/n\}, \; n \geq 1.\]
We will prove that the sets $\{L_n,\ n \geq 1\}$ form a countable local base at $x_0$. First, observe that $\displaystyle{\{x_0\} = \cap_{n=1}^{\infty} L_n}$. This follows from the assumption that $x_0$ is the strong minimum of $h$ and every minimizing sequence converges to $x_0$. Then, take an arbitrary neighbourhood $U$ of $x_0$. To the contrary, suppose that for each $n$ there exists a point $x_n \in L_n$, $x_n \not\in U$. But $\{x_n\}_n$ is a minimizing sequence and its limit point is $x_0$, so there exists $n_0$ such that $x_n \in U$ for all $n \geq n_0$, which yields a contradiction. Therefore, $X$ is a completely regular topological space with countable local base at each point.
\end{proof}

\section{Variational principle for saddle points}
Let $X$ and $Y$ be topological spaces and $f: X \times Y \rightarrow [-\infty, +\infty]$ be an extended real-valued function. Solution to the supinf problem
$$
\sup_{x \in X} \inf_{y \in Y} f(x,y),
$$
is called any point $(x_0,y_0)$ such that
\[
f(x_0,y_0)=\inf_{y\in Y} f(x_0,y)=\sup_{x\in X}\inf_{y\in Y} f(x,y),
\]
see e.g. \cite{Kend_Rev_2017}. Analogously,  solution to the infsup problem
$$
\inf_{y \in Y} \sup_{x \in X} f(x,y),
$$
is called any point $(x_0,y_0)$ such that
\[
f(x_0,y_0)=\sup_{x\in X} f(x,y_0)=\inf_{y\in Y}\sup_{x\in X} f(x,y).
\]

For a given function $f:X \times Y \rightarrow [-\infty, +\infty]$, denote
\begin{equation}\label{v}
v_f(x) := \inf_{y \in Y} f(x,y),
\end{equation}
and by $V_f$ denote the optimal value of the supinf problem, i.e. $V_f:= \sup _{x\in X}v_f(x)$. Also, denote
\begin{equation}\label{w}
w_f(y) := \sup_{x \in X}  f(x,y),
\end{equation}
and by $W_f$ denote the optimal value of the infsup problem, i.e. $W_f:= \inf _{y\in Y}w_f(x)$.

Let 
\[
\Delta_f:= W_f-V_f=\inf_{y \in Y} \sup_{x \in X}f(x,y)-\sup_{x \in X} \inf_{y \in Y}f(x,y).\]
It is clear that $\Delta_f\ge 0$.

A point $(x_0,y_0)$ is said to be a saddle point of $f$ on $X \times Y$ if
\[f(x,y_0) \leq f(x_0,y_0) \leq f(x_0,y), \forall x \in X, \forall y \in Y,\]
which is equivalent to
\[V_f=v_f(x_0)= f(x_0,y_0)=w_f(y_0)=W_f.\]
Obviously, if $(x_0,y_0)$ is a saddle point of $f$ on $X \times Y$, then  $(x_0,y_0)$ is a solution to both infsup and supinf problems for $f$, they have the same optimal values equal to $f(x_0,y_0)$, and $\Delta_f=0$. 

Note that it is possible that $\Delta_f=0$ but the function $f$ to have no saddle point on $X \times Y$. For example, consider the function $f(x,y):=x-y$ defined on $X \times Y$, where $X:=(0,1)$, $Y:=(0,1]$. It is easy to check that $\Delta_f=0$ but $f$ has no saddle point on $X \times Y$ since $(1,1) \notin X \times Y$.

We will make the following assumptions for the function $f:X \times Y \rightarrow [-\infty, +\infty]$:
\begin{itemize}
 \item[{\rm (A1)}] for any $y \in Y$ the function $f(\cdot, y)$ is upper semicontinuous;
 \item[{\rm (A2)}] the function $v_f(x)$ is bounded above in $X$ and proper as a function with values in $\R \cup \{-\infty\}$;
 \item[{\rm (A3)}] for any $x \in X$ the function $f(x,\cdot)$ is lower semicontinuous;
 \item[{\rm (A4})] the function $w_f(y)$ is bounded below in $Y$ and proper as a function with values in $\R \cup \{+\infty\}$.
\end{itemize}

Under the assumptions $(A1)$-$(A4)$, $\Delta_f$ is  finite. 

Let us recall the supinf variational principle of Kenderov and Revalski.
\begin{theo}[\cite{Kend_Rev_2017}] \label{supinf}  Let $X$ and $Y$ be completely regular topological spaces and $f:X \times Y \rightarrow [-\infty, +\infty]$ be an extended real-valued function which satisfies the assumptions $(A1)$-$(A2)$. Let $\varepsilon>0$ and $x_0 \in X$ be such that $v_f(x_0)>\sup_{x \in X} v_f(x) - \varepsilon$, and let $\delta>0$ and $y_0 \in Y$ be such that $f(x_0,y_0)<\inf_{y \in Y} f(x_0,y)+\delta$. Then, there exist continuous bounded functions $q: X \rightarrow \mathbb R_{+}$ and $p: Y \rightarrow \mathbb R_{+}$, such that $q(x_0)=p(y_0)=0$, $\|q\|_{\infty} < \varepsilon$, $\|p\|_{\infty} < \delta$ and the supinf problem
\begin{equation}\label{N1}
\sup_{x \in X} \inf_{y \in Y} \{f(x,y)-q(x)+p(y)\}
\end{equation}
 has a solution at $(x_0,y_0)$.
\end{theo}
Note that assuming (A3) the function $w_f(y)$ defined by (\ref{w}) is lower semicontinuous, so one can follow the lines of the proof of the above result in \cite{Kend_Rev_2017} to obtain the following

\begin{theo}\label{infsup}
Let $X$ and $Y$ be completely regular topological spaces and $f:X \times Y \rightarrow [-\infty, +\infty]$ be an extended real-valued function which satisfies the assumptions $(A3)$-$(A4)$. Let $\varepsilon>0$ and $y_0 \in Y$ be such that $w_f(y_0)<\inf_{y \in Y} w_f(y) + \varepsilon$, and let $\delta>0$ and $x_0 \in X$ be such that $f(x_0,y_0)>\sup_{x \in X} f(x,y_0)-\delta$. Then, there exist continuous bounded functions $h: X \rightarrow \mathbb R_{+}$ and $g: Y \rightarrow \mathbb R_{+}$, such that $h(x_0)=g(y_0)=0$, $\|h\|_{\infty} < \delta$, $\|g\|_{\infty} < \varepsilon$ and the infsup problem
\begin{equation}\label{N2}
\inf_{y \in Y} \sup_{x \in X} \{f(x,y)-h(x)+g(y)\}
\end{equation}
 has a solution at $(x_0,y_0)$.
\end{theo}

 We will prove the following minimax variational principle in completely regular topological spaces.
 
\begin{theo} \label{saddle_point} Let $X$ and $Y$ be completely regular topological spaces and $f:X \times Y \rightarrow [-\infty, +\infty]$ be an extended real-valued function which satisfies the assumptions $(A1)$-$(A4)$. Let $\varepsilon'>0$, $\varepsilon''>0$, $x_0 \in X$ and $y_0 \in Y$ be such that \[v_f(x_0)>\sup_{x \in X} v_f(x) - \varepsilon',\] \[w_f(y_0)<\inf_{y \in Y} w_f(y) + \varepsilon''.\] Then, there exist continuous bounded functions $k: X \rightarrow \mathbb R_{+}$ and $r: Y \rightarrow \mathbb R_{+}$, such that $k(x_0)=r(y_0)=0$, $\|k\|_{\infty} < 2\varepsilon'+\varepsilon''+\Delta_f$, $\|r\|_{\infty} < \varepsilon'+2\varepsilon''+\Delta_f$, and the function $f(x,y)-k(x)+r(y)$ has a saddle point at $(x_0,y_0)$.
\end{theo}

\begin{proof}
We begin with two chains of inequalities:
\begin{equation}\label {infsup_ineq}
f(x_0,y_0)-\varepsilon'' \leq \sup_{x \in X} f(x,y_0)-\varepsilon''=w_f(y_0)-\varepsilon''
<\inf_Y w_f= \inf_{y \in Y}\sup_{x \in X}f(x,y),
\end{equation}
and
\begin{equation}\label {supinf_ineq}
f(x_0,y_0)+\varepsilon' \geq \inf_{y \in Y} f(x_0,y)+\varepsilon'=v_f(x_0)+\varepsilon'
>\sup_X v_f= \sup_{x \in X}\inf_{y \in Y}f(x,y).
\end{equation}
Combining the equality 
$$\inf_{y \in Y} \sup_{x \in X}f(x,y) - \Delta_f=\sup_{x \in X} \inf_{y \in Y}f(x,y),$$
with  \eqref{infsup_ineq} and \eqref{supinf_ineq}, we get
\[f(x_0,y_0)- \varepsilon'' - \Delta_f < \inf_{y \in Y}\sup_{x \in X}f(x,y) - \Delta_f = \sup_{x \in X}\inf_{y \in Y}f(x,y)<\inf_{y \in Y} f(x_0,y)+\varepsilon',\]
hence
\[f(x_0,y_0)<\inf_{y \in Y} f(x_0,y)+\varepsilon'+\varepsilon''+\Delta_f.\]
Analogously, we get
\[f(x_0,y_0)>\sup_{x \in X} f(x,y_0)-\varepsilon'-\varepsilon''-\Delta_f.\]

Now we apply Theorem \ref{supinf} to the point $x_0$ for $\varepsilon=\varepsilon'$, and to the point $y_0$ for $\delta=\varepsilon'+\varepsilon''+\Delta_f$. Therefore, there exist continuous bounded functions  $q: X \rightarrow \mathbb R_{+}$, and $p: Y \rightarrow \mathbb R_{+}$ such that $q(x_0)=p(y_0)=0$, $\|q\|_{\infty} < \varepsilon'$, $\|p\|_{\infty} < \varepsilon'+\varepsilon''+\Delta_f$ and the supinf problem \eqref{N1} has a solution at $(x_0,y_0)$, i.e.
\begin{equation} \label{supinf_sol}
\sup_{x \in X} \inf_{y \in Y} \{f(x,y)-q(x)+p(y)\}=\inf_{y \in Y} \{f(x_0,y)-q(x_0)+p(y)\}=f(x_0,y_0).
\end{equation}

Further, we apply Theorem \ref{infsup} to the point $y_0$ for $\varepsilon=\varepsilon''$, and to the point $x_0$ for $\delta=\varepsilon'+\varepsilon''+\Delta_f$. Hence, there exist continuous bounded functions $h: X \rightarrow \mathbb R_{+}$, $g: Y \rightarrow \mathbb R_{+}$ such that $h(x_0)=g(y_0)=0$, $\|h\|_{\infty} < \varepsilon'+\varepsilon''+\Delta_f$, $\|g\|_{\infty} < \varepsilon'$ and the infsup problem \eqref{N2} has a solution at $(x_0,y_0)$, i.e.
\begin{equation} \label{infsup_sol}
\inf_{y \in Y} \sup_{x \in X} \{f(x,y)-h(x)+g(y)\}=\sup_{x \in X} \{f(x,y_0)-h(x)+g(y_0)\}=f(x_0,y_0).
\end{equation}

Let $x \in X$, $y \in Y$ be arbitrary. As $q(x) \geq 0$, $g(y) \geq 0$, from \eqref{infsup_sol} and \eqref{supinf_sol} it follows that
\[f(x,y_0)-h(x)-q(x)+p(y_0)+g(y_0)\leq f(x,y_0)-h(x)+g(y_0) \leq \]
\[ f(x_0,y_0)\leq\]
\[f(x_0,y)+p(y)-q(x_0)\leq f(x_0,y)+p(y)+g(y)-q(x_0)-h(x_0).\]
Setting $k(x):=h(x)+q(x)$ and $r(y):=p(y)+g(y)$ we get the conclusion of the theorem.
\end{proof}

Note that whenever a function $f$ satisfies the assumptions $(A1)$-$(A4)$, for any $\varepsilon'>0$ and $\varepsilon''>0$ one can always find $x_0$ and $y_0$ satisfying the assumptions of  Theorem \ref{saddle_point}.

\begin{defi}
For a function $f:X\times Y\to [-\infty,+\infty]$, such that $\Delta_f=0$ and $\varepsilon>0$ we say that $(x_0,y_0) \in X \times Y$ is an $\varepsilon$-saddle point for  $f$ if
\[v_f(x_0)>\sup_{x \in X} v_f(x) - \varepsilon/3,\] \[w_f(y_0)<\inf_{y \in Y} w_f(y) + \varepsilon/3.\]
\end{defi}

If $f$ satisfies $(A1)$-$(A4)$ and  $\Delta_f=0$, from Theorem~\ref{saddle_point} easily follows a variational principle which states that we can perturb the function $f$ by  functions with arbitrary small  norms in a way that the perturbed function has a saddle point.

\begin{theo} \label{saddle_point2} 
Let $X$ and $Y$ be completely regular topological spaces,  $f:X\times Y\to [-\infty,+\infty]$ satisfy $(A1)$-$(A4)$ and $\Delta_f=0$. Let $\varepsilon>0$,  and $(x_0,y_0) \in X \times Y$ be an $\epsilon$-saddle point for $f$. Then, there exist continuous bounded functions $k: X \rightarrow \mathbb R_{+}$, and $r: Y \rightarrow \mathbb R_{+}$, such that $k(x_0)=r(y_0)=0$, $\|k\|_{\infty} < \varepsilon$, $\|r\|_{\infty} < \varepsilon$, and the function $f(x,y)-k(x)+r(y)$ has a saddle point at $(x_0,y_0)$.
\end{theo}

If $f$ is additively separable function, i.e. $f(x,y)=f_1(x)+f_2(y)$, then $\Delta_{f+k+r}=0$ for every $k \in C(X)$ and $r \in C(Y)$. The following theorem constitutes a dense variational principle for saddle point problem.
\begin{theo}
Let $X$ and $Y$ be completely regular topological spaces and let $f:X\times Y\to [-\infty,+\infty]$ satisfy $(A1)$-$(A4)$ and $f(x,y)=f_1(x)+f_2(y)$. Then  the set $\{(k,r) \in C(X) \times C(Y):$ the function $f(x,y)+k(x)+r(y), (x,y) \in X \times Y$  has a saddle point$\}$ is a dense subset of $C(X) \times C(Y)$.
\end{theo}

\section{Well-posedness of saddle point problems}
Let $f: X \times Y \rightarrow [-\infty,+\infty]$ and $X$ and $Y$ be completely regular topological spaces. As it is defined in \cite{KL} (in the constrained case) a sequence of points $\{(x_n, y_n)\}_n \in X \times Y$ is called optimizing for the supinf  problem for $f$ if $v_f(x_n) \rightarrow V_f$ and $f(x_n,y_n) \rightarrow V_f$. Analogously,   $\{(x_n, y_n)\}_n  $ is   optimizing for the infsup  problem for $f$ if $w_f(y_n) \rightarrow W_f$ and $f(x_n,y_n) \rightarrow  W_f$.

Here we consider the following 
\begin{defi}\label{def_wp} 
A sequence $\{(x_n, y_n)\}_n \in X \times Y$ is called optimizing for the saddle point problem for $f: X \times Y \rightarrow [-\infty, +\infty]$ if
\begin{enumerate}
  \item[\textsl{1.}] $v_f(x_n) \rightarrow V_f$;
  \item[\textsl{2.}] $w_f(y_n) \rightarrow W_f$;
  \item[\textsl{3.}] $\Delta_f=0$.
\end{enumerate}
\end{defi}

Let us note that in \cite{CM} maximinimizing sequences  are considered. A sequence  $\{(x_n, y_n)\}_n \in X \times Y$ is called maximinimizing for the saddle point problem for $f$ if $w_f(y_n)-v_f(x_n) \rightarrow 0$, as $n \rightarrow \infty$. Since $v_f(x_n)=\inf_{y \in Y} f(x_n,y) \leq f(x_n,y_n) \leq \sup_{x \in X} f(x,y_n)=w_f(y_n)$, it is clear that every optimizing sequence for the saddle point problem is maximinimizing. Moreover, any optimizing sequence for the saddle point problem for $f$ is optimizing for the supinf and infsup problems for $f$.

\smallskip

The supinf (resp. infsup) problem for the function $f$ is well-posed if any optimizing for it sequence converges to its unique solution.

Moreover, the supinf  problem for the function $f$ is sup-well-posed  if the problem $\sup_{x \in X} v_f(x)$  is well-posed, see \cite{Kend_Rev_2017}. In such a case, the unique point realizing the maximum is called sup-solution.

Analogously, infsup  problem for the function $f$ is inf-well-posed  if the problem $\inf_{y \in Y} w_f(y)$  is well-posed and the unique point realizing the minimum is called inf-solution.

\begin{defi}\label{def_wpsp}
The saddle point problem for $f: X \times Y \rightarrow [-\infty, +\infty]$ is well-posed if every optimizing for it sequence  converges to the unique solution  of the problem.
\end{defi}

The saddle point problem for a function $f$ is well-posed if and only if the supinf and infsup problems for $f$ are sup-well-posed and inf-well-posed, respectively, and $\Delta_f=0$. Well-posedness of the saddle point problem for $f$ does not entail well-posedness of the corresponding supinf and infsup problems. See Remark 3-1 in~\cite{CM}. 

\newpage

If $f$ is such that $\Delta_f=0$, then for $f$:
\[\{(x_n,y_n)\}_n\text{is an optimizing sequence for the saddle point problem}\]
\[\Updownarrow\]
\[\{(x_n,y_n)\}_n\text{is an optimizing sequence for both supinf and infsup problems,}\]
and
\[\text{the supinf and infsup  problems for}f\text{are well-posed}\]
\[\Downarrow\]
\[\text{the saddle point problem for}f \text{is well-posed}\]
\[\Updownarrow\]
\[\text{the supinf problem  is sup-well-posed and the infsup problem is inf-well-posed.}\]

The next result is about perturbations for which saddle point problem is  well-posed in the sense of Definition~\ref{def_wpsp}.

\begin{theo}\label{saddle_point_wp} 
Let for $X,Y$ and $f: X \times Y \rightarrow [-\infty, +\infty]$   the assumptions in Theorem~\ref{saddle_point} hold and let $k$ and $r$ be the functions from its conclusion. Suppose that $x_0$ has a countable local base in $X$ and $y_0$ has a countable local base in $Y$. Then, for arbitrary
 $\delta>0$ there exist continuous bounded functions $k': X \rightarrow \mathbb R_{+}$ and $r': Y \rightarrow \mathbb R_{+}$, such that $k(x_0)=r(y_0)=0$, $\|k'\|_{\infty} < \delta$, $\|r'\|_{\infty} < \delta$, and for the function $g(x,y):=f(x,y)-k(x)+r(y)-k'(x)+r'(y)$
\begin{enumerate}
    \item[(a)] the supinf problem is sup-well-posed with unique sup-solution at $x_0$;
    \item[(b)] the infsup problem  is inf-well-posed with unique inf-solution at $y_0$;
    \item[(c)] the saddle point problem  is well-posed with unique solution at $(x_0,y_0)$.
\end{enumerate}
\end{theo}

\begin{proof}
We follow the idea of the proof of Proposition 2.10 in \cite{Kend_Rev_2017}. Consider countable  local bases of open nested neighbourhoods $\{U_n,\ n \geq 1\}$ of $x_0$ and $\{V_n,\ n \geq 1\}$  of $y_0$, respectively. For each fixed $n \geq 1$, let $h_n : X \rightarrow [0, 1]$ and $g_n : Y \rightarrow [0, 1]$ be  continuous functions such that
\[h_n(x_0) = 0 \text{ and } h_n(X \setminus U_n) \equiv 1,\]
\[g_n(y_0) = 0 \text{ and } g_n(Y \setminus V_n) \equiv 1.\]

Define the functions
\[  k'(x):=\delta \sum_{n=1}^\infty \frac{1}{2^n}h_n(x), \; x \in X,\]
\[ r'(y):=\delta \sum_{n=1}^\infty \frac{1}{2^n}g_n(y), \; y \in Y.\]

The functions $k'$ and $r'$ are  continuous bounded functions with values in $[0,\delta]$. For all $x \in X$ and $y \in Y$, $x \neq x_0$ and $y \neq y_0$, $k'(x)>0$ and $r'(y)>0$. Furthermore, from $k(x_0)=k'(x_0)=0$, $r(y_0)=r'(y_0)=0$ and the fact that $(x_0,y_0)$ is a saddle point for $f(x,y)-k(x)+r(y)$, then for all $x \in X$ and $y \in Y$, $x \neq x_0$ and $y \neq y_0$,
\[f(x,y_0)-k(x)-k'(x)+r(y_0)+r'(y_0)<\]
\[f(x_0,y_0)-k(x_0)-k'(x_0)+r(y_0)+r'(y_0)\]
\[<f(x_0,y)-k(x_0)-k(x_0)+r(y)+r'(y),\]
and therefore $(x_0,y_0)$ is a solution for the saddle point problem, supinf problem and infsup problem for the function $g(x,y)$ in $X \times Y$. Observe that $(x_0,y_0)$ is also a solution for the saddle point problem, supinf problem and infsup problem for the functions $f(x,y)-k(x)+r(y)+r'(y)$ and $f(x,y)-k(x)-k'(x)+r(y)$ in $X \times Y$.

\smallskip

To prove (a) let us consider
$$
\begin{array}{ll}
v_{f-k+r+r'}(x):=\inf_Y \{f(x,y)-k(x)+r(y)+r'(y)\},\\
v_{g}(x):=\inf_Y \{f(x,y)-k(x)-k'(x)+r(y)+r'(y)\}.
\end{array}
$$
Having in mind that $k(x_0)=k'(x_0)=0$ and $r(y_0)=r'(y_0)=0$ and $(x_0,y_0)$ is a solution for the supinf problem for the functions $f(x,y)-k(x)+r(y)+r'(y)$ and $g(x,y)$ in $X \times Y$, we have that
$$
\begin{array}{ll}
v_{f-k+r+r'}(x_0)=v_{g}(x_0)=f(x_0,y_0).
\end{array}
$$
Suppose that $\{x_n\}_n$ is an optimizing sequence for $v_{g}$ and $\{(x_n)\}_n$ does not converge to $x_0$. Then, there would be some integer
$n_0 \geq 1$ and a subsequence of $\{x_n\}_n$ (without renumbering), such that $x_n \notin U_{n_0}$ for every $n$. Further,  $\sup_X v_{g} = \sup_X v_{f-k+r+r'}$ and $v_{g}(x) \leq v_{f-k+r+r'}(x)$,
and since $k'$ has nonnegative values, $k'(x_n)$ converges to 0. But this contradicts $x_n \notin U_{n_0}$ for any $n$ since the latter would imply that $k'(x_n)\ge \delta \sum_{m=n_0}^\infty(1/2^m)>0$ for all~$n$.

\smallskip

Proof of (b) is similar to the proof of (a).

\smallskip

To establish (c) let us observe that from the proofs of (a) and (b) it follows that the supinf and infsup problems for the function $g$ are sup-well-posed and inf-well-posed with unique solutions at $x_0$ and $y_0$, respectively. Moreover $\Delta_{g}=0$ and, therefore, the saddle point problem for $g$ is well-posed with unique solution at $(x_0,y_0)$.
\end{proof}

For a given function $f: X \times Y \rightarrow [-\infty, +\infty]$ denote by $S_f: C(X) \times C(Y) \rightrightarrows X \times Y$ the correspondence which assigns to every couple of functions $s \in C(X)$ and $u \in C(Y)$ the (possibly empty) set of solutions $(x_0,y_0)$ to the saddle point problem for the function $f(x,y)+s(x)+u(y)$. We follow  the proof of Theorem 3 in \cite{GKR} to get the following result.

\begin{theo} \label{wellposed_su} 
Let $X$ and $Y$ be completely regular topological spaces and $f:X \times Y \rightarrow [-\infty, +\infty]$ be an extended real-valued function which satisfies $(A1)$-$(A4)$. The mapping $S_f$ is single-valued and upper semicontinuous at $(s,u) \in C(X) \times C(Y)$ if and only if the saddle-point problem for the function $f(x,y)+s(x)+u(y)$ is well-posed.
\end{theo}

\begin{proof}
$\Rightarrow)$ Let $\{(x_n,y_n)\}_n \in X \times Y$ be an optimizing sequence for the saddle point problem for $f(x,y)+s(x)+u(y)$. According to the definition of such a sequence, it holds that:
\begin{enumerate}
  \item $v_{f+s+u}(x_n)  \rightarrow V_{f+s+u}$;
  \item $w_{f+s+u}(y_n) \rightarrow W_{f+s+u}$;
  \item $\Delta_{f+s+u}=0$.
\end{enumerate}
If $S_f(s,u)=\{(x_0,y_0)\}$, so $\Delta_{f+s+u}=0$. Suppose that $\{(x_n,y_n)\}_n$ does not converge to $(x_0,y_0)$. Then, there exist open sets $U$ and $V$, such that $x_0 \in U$, $y_0 \in V$ and a subsequence (without renumbering) $\{(x_n,y_n)\}_n$ such that $(x_n,y_n) \notin U \times V$ for all~$n$.

\smallskip

From the upper semicontinuity of $S_f(s,u)$ there exists $\varepsilon>0$, such that $\Vert s'-s \Vert_\infty<\varepsilon$, $\Vert u'-u \Vert_\infty<\varepsilon$, $s' \in C(X)$, and $u' \in C(Y)$ imply $S_f(s',u') \subset U \times V$.

\smallskip

Let $n$ be so large that $V_{f+s+u}-v_{f+s+u}(x_n) < \varepsilon/3$ and $w_{f+s+u}(y_n)-W_{f+s+u} < \varepsilon/3$. Applying Theorem~\ref{saddle_point2} for the $\varepsilon$-saddle point $(x_n,y_n)$ of the function $f(x,y)+s(x)+u(y)$ one obtains  functions $s_n \in C(X)$ and $u_n \in C(Y)$, such that $\Vert s_n \Vert_{\infty} < \varepsilon$, $\Vert u_n \Vert_{\infty} < \varepsilon$ and $(x_n,y_n)$ is a solution to the supinf problem for the function $f(x,y)+s(x)-s_n(x)+u(y)+u_n(y)$. But $\Vert s+s_n-s \Vert_{\infty}<\varepsilon$, $\Vert u-u_n-u \Vert_{\infty}<\varepsilon$ and $(x_n,y_n) \in S_f(s-s_n,u+u_n)$ which contradicts $(x_n,y_n) \notin U \times V$.

\smallskip

$\Leftarrow)$ Suppose that the saddle point problem for the function $f(x,y)+s(x)+u(y)$ is well-posed with unique solution $(x_0,y_0)$. Hence, $\Delta_{f+s+u}=0$ and $S_f(s,u)$ is single-valued. Suppose that $S_f$ is not upper semicontinuous at $(s,u)$. Then there would  exist open neighbourhoods $U$ of $x_0$ and $V$ of $y_0$, respectively, such that for every $n \geq 1$ there will be $s_n \in C(X)$ and $u_n \in C(Y)$ with $\Vert s_n-s\Vert_{\infty} <1/n$ and $\Vert u_n-u\Vert_{\infty} <1/n$, such that $S_f(s_n,u_n)\notin U \times V$, i.e. for every $n$ there will be $(x_n,y_n) \in S_f(s_n,u_n) \setminus U \times V$.

\smallskip

Observe that $\Delta_{f+s_n+u_n}=0$, and 
$$
\begin{array}{ll}
v_{f+s_n+u_n}(x_n)&:=\inf_{y \in Y}\{f(x_n,y)+s_n(x_n)+u_n(y)\}\\
& \; =\sup_{x \in X} \{f(x,y_n)+s_n(x)+u_n(y_n)\}=:w_{f+s_n+u_n}(y_n).
\end{array}
$$
As $\{s_n\}_n$ and $\{u_n\}_n$ converge uniformly on $X$ and $Y$ to $s$ and $u$, respectively, then for every $\varepsilon>0$ we can find $n_0$, such that, for every $n \geq n_0$:
\[|v_{f+s_n+u_n}(x_n)-v_{f+s+u}(x_n)| <\varepsilon,\]
\[|w_{f+s_n+u_n}(y_n)-w_{f+s+u}(y_n)| <\varepsilon.\]
Therefore $v_{f+s+u}(x_n)$ and $w_{f+s+u}(y_n)$ are close eventually and $\{(x_n,y_n)\}_n$ is an optimizing sequence for $f(x,y)+s(x)+u(y)$. Since the saddle point problem for $f(x,y)+s(x)+u(y)$ is well-posed with unique solution $(x_0,y_0)$, the sequence $\{(x_n,y_n)\}_n$ converges to $(x_0,y_0)$. The latter contradicts the assumption $(x_n,y_n) \in S_f(s_n,u_n) \setminus U \times V$.
\end{proof}

If we consider the correspondence $\widetilde{S}_f: C(X \times Y) \rightrightarrows X \times Y$  which assigns to each function $z \in C(X \times Y)$ the (possibly empty) set of solutions $(x_0,y_0)$ to the saddle point problem for the function $f(x,y)+z(x,y)$, and follow the lines of the proof of Theorem~\ref{wellposed_su} we will obtain

\begin{theo}\label{wellposed_z} 
Let $X$ and $Y$ be completely regular topological spaces and $f:X \times Y \rightarrow [-\infty, +\infty]$ be an extended real-valued function which satisfies the assumptions$(A1)$-$(A4)$. The mapping $\widetilde{S}_f$ is single-valued and upper semicontinuous at $z \in C(X \times Y)$ if and only if the saddle-point problem for the function $f(x,y)+z(x,y)$ is well-posed.
\end{theo}

\medskip

\noindent\textbf{Acknowledgements.} 
The authors express their sincere gratitude to Prof. Petar Kenderov for suggesting the idea to consider variational principles for saddle point problems in completely regular topological spaces and for his support and encouragement.

\end{document}